\documentclass[11pt]{amsart}
\usepackage[T1]{fontenc}
\usepackage[utf8]{inputenc}
\usepackage{lmodern}
\usepackage{amsmath,amssymb,amsthm}
\usepackage[margin=1.15in]{geometry}
\usepackage{microtype}
\usepackage[colorlinks=true,linkcolor=blue,citecolor=blue,urlcolor=blue]{hyperref}

\newtheorem{theorem}{Theorem}[section]

\newtheorem{lemma}[theorem]{Lemma}

\newtheorem{conjecture}[theorem]{Conjecture}
\theoremstyle{remark}
\newtheorem{remark}[theorem]{Remark}

\DeclareMathOperator{\GL}{GL}

\DeclareMathOperator{\Gal}{Gal}

\newcommand{\Mat}{\mathrm M}

\title[Commuting graphs of matrix rings of even size]{Finite matrix rings of even size are determined by their commuting graphs}
\author{Hamid Shahverdi}
\author{Vahid Shahverdi}

\date{}
\begin{document}
\begin{abstract}
The commuting graph of a noncommutative ring has its noncentral elements as vertices, with two distinct vertices adjacent when they commute. Akbari, Ghandehari, Hadian and Mohammadian conjectured that a full matrix ring over a finite field is determined, among rings with identity, by its commuting graph. In this paper, we prove the conjecture for every even matrix size.
\end{abstract}
\maketitle

\section{Introduction}

Let $R$ be a noncommutative ring with center $Z(R)$. Its \emph{commuting graph}, denoted by $\Gamma(R)$, is the simple graph with vertex set $R\setminus Z(R)$ in which two distinct vertices $x,y$ are connected if $xy=yx$. Although the commuting graph only preserves the commutativity relation, it is interesting to know which rings it determines up to isomorphism.

For rings of matrices over finite fields, Akbari, Ghandehari, Hadian and Mohammadian \cite{AGHM} proposed the following conjecture in 2004.

\begin{conjecture}%[Akbari--Ghandehari--Hadian--Mohammadian]
\label{conj}
Let $F$ be a finite field and let $m\geq 2$. If $R$ is a ring with identity such that
\begin{equation}\label{eq:graph_iso}
\Gamma(R)\cong\Gamma(\Mat_m(F)),
\end{equation}
then $R\cong\Mat_m(F)$.
\end{conjecture}

Abdollahi \cite{A} showed that \eqref{eq:graph_iso} determines the order of the ring, i.e., $|R|=|\Mat_m(F)|$. Subsequently, the conjecture was established for $m=2$ by Mohammadian \cite{M} and later for $m=2^a3^b$, with $a\geq 1$ and $b\geq 0$, by Dorbidi \cite{D}. For $m=3$, Dorbidi and Manaviyat \cite{DM} proved the conjecture with extra assumption that $Z(R)$ is a field. In this paper, as the title suggests, we resolve the conjecture for every even matrix size.

\begin{theorem}\label{thm:main}
Let $F$ be a finite field and let $n\geq1$. If $R$ is a ring with identity such that
\begin{equation}
\Gamma(R)\cong\Gamma(\Mat_{2n}(F)),
\end{equation}
then $R\cong\Mat_{2n}(F)$.
\end{theorem}

Our proof idea uses centralizers isomorphic to $\Mat_2(E)$, where $[E:F]=n$. The Skolem--Noether theorem, together with the conjugation action of $1+J(R)$ on these centralizers, gives us a bound on the dimension of the Jacobson radical $J(R)$. We then apply the double centralizer theorem in $R/J(R)$, which shows that this bound forces $J(R)=0$ and obtain the desired isomorphism.

\section{Preliminaries}

 For $X\subseteq R$, we denote its centralizer by $C_R(X)$ and consider $C_R(x)$ when $X=\{x\}$.

Our first lemma is a basic observation and establishes the relations between centralizers coming from a graph isomorphism.

\begin{lemma}\label{lem:centralizers}
Let $\phi:\Gamma(T)\to\Gamma(U)$ be a graph isomorphism. For every $x\in T\setminus Z(T)$, the restriction of $\phi$ gives rise to an isomorphism $\Gamma(C_T(x))\cong\Gamma(C_U(\phi(x)))$. Moreover, the correspondence $C_T(x)\mapsto C_U(\phi(x))$ is a bijection between the distinct centralizers of noncentral elements in $T$ and $U$.
\end{lemma}

\begin{proof}
The closed neighborhood of $x$ is $C_T(x)\setminus Z(T)$ and among its vertices, those adjacent to all the others are exactly the elements of $Z(C_T(x))\setminus Z(T)$. We delete these vertices to get $\Gamma(C_T(x))$. The same argument applies to $\phi(x)$, and $\phi$ takes the deleted vertices in one neighborhood to the deleted vertices in the other. Its restriction therefore gives the desired isomorphism in our lemma.

For the second part, we recover $C_T(x)$ from the closed neighborhood of $x$ by adding $Z(T)$. Thus, distinct centralizers correspond to distinct closed neighborhoods, and the same holds true for $U$. The bijection between closed neighborhoods given by $\phi$ therefore gives the desired bijection between centralizers.
\end{proof}

The next lemma is the key to our main theorem. There, we use the Skolem--Noether theorem and the double centralizer theorem \cite{R} for the proof of the main theorem, which we recall them briefly here.  

Let $B$ be a finite-dimensional central simple $F$-algebra. Skolem--Noether says that every $F$-isomorphism $\varphi$ between simple unital $F$-subalgebras of $B$ has the form $\varphi(x)=uxu^{-1}$ for some invertible element $u\in B$. We use this theorem for ring of matrices. In particular, two simple unital
$F$-subalgebras of $\Mat_m(F)$ are conjugate by an element of $\GL_m(F)$ whenever they are $F$-isomorphic.

The double centralizer theorem states that if $B$ is
a finite-dimensional central simple unital $F$-algebra and
$A\subseteq B$ is a simple unital $F$-subalgebra, then
\begin{equation}\label{eq:double}
    \dim_F A\,\dim_F C_B(A)=\dim_F B.
\end{equation}

\begin{lemma}\label{lem:matrix}
Let $F$ be a finite field, $E/F$ be a field extension of degree $n>1$, and $B=\Mat_{2n}(F)$. If $P \in M_n(F)$ is a companion matrix of an irreducible polynomial of degree $n$, then $C_B(Q)\cong\Mat_2(E)$, where $Q = \operatorname{diag}(P,P)$. Moreover, all centralizers in $B$ isomorphic to $\Mat_2(E)$ are conjugate under $\GL_{2n}(F)$, and their normalizers are isomorphic to $\GL_2(E)\rtimes\Gal(E/F)$.
\end{lemma}
\begin{proof}
The centralizer of the given companion matrix $P$ is precisely $F[P]$. Since its characteristic polynomial is irreducible, we get the $F$-algebra isomorphism $F[P]\cong E$. Now, a $2 \times 2$ block matrix $X = [X_{ij}]_{i,j=1,2} \in M_{2n}(F)$ commutes with $Q$ if and only if each block $X_{ij}$ commutes with $P$. Hence, the centralizer of $Q$ is $C_B(Q)=\Mat_2(F[P])\cong\Mat_2(E)$.

Now, let $A\subseteq B$ be another centralizer isomorphic to $\Mat_2(E)$. We wish to show it is conjugate to $C_B(Q)$. To see this, note that since $Z(A)$ contains $F$ and is a finite field of degree $n$ over $F$, we have $A\cong\Mat_2(E)\cong C_B(Q)$ as $F$-algebras. By the Skolem--Noether theorem, an $F$-isomorphism from $C_B(Q)$ to $A$ extends to an inner automorphism of $B$, so there exists $u\in\GL_{2n}(F)$ such that $uC_B(Q)u^{-1}=A$, which proves the second part of the lemma.

For the last part, let $N$ be the normalizer of $C_B(Q)$ in $\GL_{2n}(F)$. Note that every element of $N$ is preserving the center $Z(C_B(Q))= \{\operatorname{diag}(S,S):S\in F[P]\} \cong E$ by conjugation and fixing the scalar matrices $\lambda I_{2n}$ with $\lambda \in F$, so its action on $E \cong Z(C_B(Q))$ is clearly an $F$-automorphism of $E$. Hence, we get a group homomorphism $\pi:N\to\Gal(E/F)$. Its kernel consists exactly of the invertible matrices commuting with every element of $Z(C_B(Q))$. Now, every element of $C_B(Q)$ commutes with its center, while every element commuting with this center also commutes with $Q\in Z(C_B(Q))$, so $C_B(Z(C_B(Q)))=C_B(Q)$. Therefore, the kernel is isomorphic to $\GL_2(E)$.

We next show that every $\sigma\in\Gal(E/F)$ is induced by an element of $N$. To this end, let us first identify $E$ with $F[t]/(f(t))$, where $f$ is the characteristic polynomial of $P$. In the basis $1,t,\ldots,t^{n-1}$, multiplication by $t$ is represented by $P$, so the matrices in $F[P]$ represent multiplication by elements of $E$. Let $T_\sigma$ be the matrix of the $F$-linear map $\sigma$ in this basis. For $s\in E$, let $M_s\in F[P]$ represent multiplication by $s$. Since $\sigma(s\,\sigma^{-1}(v))=\sigma(s)v$ for every $v\in E$, we have $T_\sigma M_sT_\sigma^{-1}=M_{\sigma(s)}$. Thus conjugation by $U_\sigma=\operatorname{diag}(T_\sigma,T_\sigma)$ applies $\sigma$ to each entry of $\Mat_2(E)$ under our identification with $\Mat_2(F[P])$. Hence, $U_\sigma$ normalizes $C_B(Q)$ and acts as $\sigma$ on its center, so $\pi(U_\sigma)=\sigma$.

Finally, the matrices $U_\sigma$ form a subgroup on which the restriction of $\pi$ to this subgroup is an isomorphism onto $\Gal(E/F)$, so $N\cong\GL_2(E)\rtimes\Gal(E/F)$.

\end{proof}

\section{Proof of Theorem \ref{thm:main}}
\begin{proof}
The base case $n=1$ follows from \cite{M}, so we assume $n>1$. Now, let $|F|=q=p^\ell$, $B=\Mat_{2n}(F)$, and  $\phi:\Gamma(B)\to\Gamma(R)$ be a graph isomorphism. By \cite{A}, together with equality of the numbers of vertices, the orders of the ring and its center are
\begin{equation}\label{eq:orders}
|R|=q^{4n^2},\qquad |Z(R)|=q.
\end{equation}

We choose $Q\in B$ as in Lemma \ref{lem:matrix}, so that $C_B(Q)\cong\Mat_2(E)$ with $[E:F]=n$. By Lemma \ref{lem:matrix}, we thus have
\begin{equation}
\Gamma(C_R(\phi(Q)))\cong\Gamma(\Mat_2(E)).
\end{equation}
By \cite{M}, we get $C_R(\phi(Q))\cong\Mat_2(E)$ and note that its center is a field containing $Z(R)$, so $Z(R)$ is a finite domain and hence a field too. By \eqref{eq:orders}, we identify $Z(R)$ with $F$ and consider $R$ as a central $F$-algebra of dimension $4n^2$.

Let $J=J(R)$ be the Jacobson radical of $R$ and let
\begin{equation}
\mathfrak C=\{C_R(x):x\notin Z(R),\ C_R(x)\cong\Mat_2(E)\}.
\end{equation}
We first want to show that $1+J$ acts freely on $\mathfrak C$ by conjugation. To this end, take  $C=C_R(x)\in\mathfrak C$ and since $C$ is simple and $1\notin J$, the ideal $C\cap J$ of $C$ must be the zero ideal. If $u\in 1+J$ normalizes $C$, then $ucu^{-1}-c\in C\cap J =0$ for every $c\in C$ and thus $ucu^{-1}=c$. Thus, $u$ commutes with $c$ and in particular, it commutes with $x\in C$, so $u\in C$. Moreover, $u-1 \in C \cap J=0$, hence, $u=1$. Therefore, the action is free, which, as a consequence,
\begin{equation}\label{eq:divisibility}
|J|=|1+J|\mid |\mathfrak C|.
\end{equation}

By Lemma \ref{lem:matrix} and \cite{M}, $\phi$ gives a bijection between $\mathfrak C$ and the centralizers in $B$ isomorphic to $\Mat_2(E)$. Lemma~\ref{lem:matrix} shows that these form one conjugacy orbit, with normalizer of order $n|\GL_2(q^n)|$. Consequently,
\begin{equation}
|\mathfrak C| =\frac{|\GL_{2n}(q)|}{n|\GL_2(q^n)|} =\frac{q^{n(2n-1)}\prod_{i=1}^{2n}(q^i-1)} {nq^n(q^n-1)(q^{2n}-1)} =\frac{q^{2n(n-1)}}{n} \prod_{\substack{1\leq i\leq 2n\\ i\neq n,\,2n}}(q^i-1).
\end{equation}
Since $q=p^\ell$, then none of the factors $q^i-1$ is divisible by the prime $p$. Therefore, $v_p(|\mathfrak C|)=2\ell n(n-1)-v_p(n)$, where $v_p(k)$ denotes the exponent of $p$ in $k$. Now, since $|J|=q^{\dim_F J}$ and
$|J|\mid|\mathfrak C|$ by \eqref{eq:divisibility}, so we have the following upper bound
\begin{equation}
\ell\dim_F J \leq 2\ell n(n-1)-v_p(n) \leq 2\ell n(n-1),
\end{equation}
hence,
\begin{equation}\label{eq:bound}
\dim_F J\leq 2n(n-1)<2n^2.
\end{equation}

Now, consider $S=R/J$ and take the same $C=C_R(x)\in\mathfrak C$ as before. We claim that the image of $C$ in $S$ is precisely $C_S(\bar x)$, where $\bar x=x+J$. Indeed every element of $C$ commutes with $x$, thus its image belongs to $C_S(\bar x)$. For the reverse inclusion, take the following $F$-linear map
\begin{equation}
J\longrightarrow J,\qquad j\longmapsto jx-xj.
\end{equation}
Its kernel is $C\cap J=0$, so it is bijective. If $b+J\in C_S(\bar x)$, then $bx-xb\in J$. There is therefore $j\in J$ such that $jx-xj=bx-xb$, which means that $b-j\in C$. Since $(b-j)+J=b+J$, this proves the claim which, together with $C\cap J=0$, gives
$C_S(\bar x)\cong C\cong\Mat_2(E)$.

Next, we wish to show that the semisimple ring $S$ is in fact simple. If not, a decomposition of $S$ into at least two simple factors would give a decomposition of $C_S(\bar x)$ into at least two nonzero factors, since the centralizer in each factor contains its identity. This contradicts the simplicity of $C_S(\bar x) \cong \Mat_2(E)$, which we established earlier.

We turn to proving $J=0$. Since $F\cap J=0$, we identify $F$ with its image in $S$ and let $d=[Z(S):F]$. Now, $Z(S)$ is contained in $Z(C_S(\bar x))\cong E$, so $d\mid n$. Note that since $S$ is a central simple algebra over $Z(S)$, we
can then apply the double centralizer theorem to get
\begin{equation}
\label{eq:S_Z(S)}
\dim_{Z(S)}S =\dim_{Z(S)}C_S(\bar x)\, \dim_{Z(S)}Z(C_S(\bar x)) =\frac{4n}{d}\frac{n}{d}.
\end{equation}
From \ref{eq:S_Z(S)}, we deduce that $\dim_F S=4n^2/d$, and therefore
\begin{equation}
\dim_F J=\dim_F R-\dim_F S =4n^2\left(1-\frac{1}{d}\right).
\end{equation}
Note that if $d\geq2$, then we get $\dim_F J\geq2n^2$, which contradicts \eqref{eq:bound}, so we must have $d=1$ and consequently $J=0$.

Finally, we conclude that $R$ is a finite simple ring (since $J=0$ and we have shown that $S=R/J$ is simple) with center $F$, so the Artin--Wedderburn theorem and Wedderburn's little theorem give $R\cong\Mat_t(F)$ for some positive integer $t$. Since, $\dim_F R=4n^2$, this implies $t=2n$, which completes the proof.

\end{proof}

\begin{remark}
We proved Conjecture \ref{conj} for matrices of even size using the known case of size $2$. Indeed, our method extends to every multiple of a prime once the conjecture is known for that prime size over every finite field. In particular, if the conjecture is proved for every prime size over every finite field, we can surely deduce it for every size. 
\end{remark}

\section*{Declaration of generative AI use}
The authors used ChatGPT (OpenAI) for grammar checking, language editing, and assistance with proof checking. The main ideas and proofs are the authors' own.

\end{document}